\documentclass{cimart}

\newcommand{\orcid}[1]{\href{https://orcid.org/#1}{\textcolor[HTML]{A6CE39}{\aiOrcid}}}

\usepackage[capitalize]{cleveref}
\usepackage{quiver}
\tikzcdset{scale cd/.style={every label/.append style={scale=#1},cells={nodes={scale=#1}}}}

\usepackage[normalem]{ulem}

\DeclareMathOperator{\Act}{Act}
\DeclareMathOperator{\Aut}{Aut}

\DeclareMathOperator{\id}{id}
\newcommand{\PostL}{\text{\bf PostLie}}
\newcommand{\PreL}{\text{\bf PreLie}}
\newcommand{\Skb}{\text{\bf Skb}}
\DeclareMathOperator{\Split}{Split}
\DeclareMathOperator{\Sym}{Sym}
\DeclareMathOperator{\Soc}{Soc}
\newcommand{\Z}{\mathbb{Z}}
\newcommand{\N}{\mathbb{N}}
\newcommand{\Set}{\mathbf{Set}}
\newcommand{\F}{\mathcal{F}}

\newcommand{\C}{\mathcal{C}}
\newcommand{\V}{\mathcal{V}}
\renewcommand{\to}[1][\,]{\xrightarrow{#1}}

\newcommand{\mybinom}[2]{\resizebox{!}{12pt}{$\displaystyle \binom{#1}{#2}$}}

\theoremstyle{definition}
\newtheorem{question}[definition]{Question}
\newtheorem{examples}[definition]{Examples}

\title{Action accessibility in the variety of skew braces
    }

\authors{Andrea Albano, Paola Stefanelli
    }

\authorinfo[
    A. Albano]{
     University of Salento, Italy}{
    andrea.albano@unisalento.it
    }

\authorinfo[
    P. Stefanelli]{
   University of  Salento, Italy}{
    paola.stefanelli@unisalento.it
    }

\abstract{
    In this paper, we answer negatively a question posed in the context of the Oberwolfach Mini-Workshop ``The Yang--Baxter Equation and Representations of Braid Groups'' \cite{CoPlRoVe25} regarding the existence of split extensions classifiers in the category of skew braces.
    To this end, we show that the variety of skew braces is not action accessible by investigating the categorical notion of centralizer in the sense of Huq \cite{Huq68}.
    In light of their intrinsic relationship with skew braces, an analogous result is proven for the variety of post-Lie algebras over a fixed field.
    }



\VOLUME{35}
\YEAR{2027}
\ISSUE{2}
\NUMBER{3}
\DOI{https://doi.org/10.46298/cm.18819}
\licence{CC BY 4.0}
\editinfo{July 21, 2026}{September 15, 2026}{Ivan Kaygorodov and Maxime Fairon}

\acknowledgments{Both authors are members of GNSAGA (INdAM) and the non-profit association ADV-AGTA. This work was supported by the Dipartimento di Matematica e Fisica ``Ennio De Giorgi" -- Università del Salento. A.~Albano was supported by a scholarship financed by the Ministerial Decree no.~118/2023, based on the NRRP -- funded by the European Union -- NextGenerationEU -- Mission 4. We thank the referees for their careful reading and for the constructive comments which have helped to improve the quality of the paper.
In particular, two comments allowed us to refine a few final points in our main result.
Moreover, a third comment allowed us to state Question 3.7 in a more suitable form.}

\begin{document}

\section*{Introduction}

Skew brace theory finds its roots in early foundational works \cite{EtSchSol99,LuYanZhu00,Sol00} related to the investigation of non-degenerate solutions to the set-theoretic Yang--Baxter equation \cite{Drinfeld92}.
These structures first arose in a commutative setting in \cite{Rump07} as a generalization of Jacobson radical rings and were later extended in \cite{GuVe17}.
Following the formulation given in \cite{CedFerOkn14, GuVe17}, a skew (left) brace is a set $B$ equipped with two group operations $+$ and $\circ$ satisfying the following identity, for all $a,b,c \in B$:
\begin{align*}
        a \circ (b+c) = a \circ b - a + a \circ c \,. 
\end{align*}

Despite its evidently group-theoretical flavor, the study of such structures has deeply benefited from categorical and universal algebraic considerations.
For instance, in \cite{BonJed23} the authors employed the Freese--McKenzie commutator theory to obtain a natural notion of (central) nilpotency that turned out to be stronger than previous definitions already present in the literature; see \cite{Rump07,Sm18}. 
In a similar vein, protomodular aspects of the category $\Skb$ of skew braces were investigated by the authors of \cite{BouFacPom23} who proved that the so-called \emph{Smith is Huq} condition holds in $\Skb$ and provided  
an algebraic description of the commutator of two ideals in a skew brace.
These results have been the basis for further studies as found, for example, in 
\cite{Bal25}, where a systematic investigation of central nilpotency was carried out, and in
\cite{Bal24solv} where a novel definition of solvability for skew braces was formulated.

Hopf-theoretic methods have also played an important role in enriching the theory of skew braces, finding their natural setting in Hopf braces, introduced in \cite{AngGalVen17} and
for which similar considerations as above hold;
see \cite{AgoChi26,GranScia26} and references therein.
An instance of this interaction is given by the notion of a skew brace representation, which, in analogy with modules over a Hopf brace, was proposed in \cite{Zhu22} and has received attention especially in connection with the representation theory of groups; see for example \cite{KozTsang25,LetVen24,RatUd26x}.

\smallskip

The main motivation of the current work originates from the necessity of investigating whether the representation theory of skew braces can be put in a categorical-algebraic framework.
As one would expect, 
a representation of a skew brace $B$ should be equivalent, if possible, to the datum of a natural skew brace of linear automorphisms $S(V) \subseteq \text{GL}_k(V)$ acting naturally on a $k$-vector space $V$ through a skew brace homomorphism $B \to S(V)$.
Here, naturality refers to the fact that there should exist a bijective correspondence between morphisms $B \to S(V)$ as above and split short exact sequences of skew braces 
$\begin{tikzcd}[ampersand replacement=\&,column sep=small]
	0 \& V \& {V \rtimes B} \& B \& 0
	\arrow[from=1-1, to=1-2]
	\arrow[tail, from=1-2, to=1-3]
	\arrow[shift left, two heads, from=1-3, to=1-4]
	\arrow[shift left, from=1-4, to=1-3]
	\arrow[from=1-4, to=1-5]
\end{tikzcd}$.
In light of the semi-abelianity of $\Skb$, these considerations lead us to address the question of whether \emph{action representability} holds in this category, in the sense of \cite{BorJanKel05}.
It is indeed through this lens that we interpret a question posed during ``The Yang--Baxter Equation and Representations of Braid Groups'' \cite[p.~2635]{CoPlRoVe25}, which asked 
\begin{center}
``Does there exist an internal Aut-object $[X]$ in the category of skew braces that classifies split extensions (or actions) via a universal property?''
\end{center}
where $[X]$ is $S(V)$ as above. 
We respond negatively to this question precisely by verifying that $\Skb$ is not action-accessible, thus not action representable, showing in this way that the above heuristic for skew brace representations is not valid.

The paper is divided into three sections.
The first section contains a brief introduction to semi-abelianity and its characterization for varieties in the sense of Universal Algebra, along with a concise overview of the relationship between action representability and action accessibility.
The second section collects basic facts about skew braces that are needed to prove our main result \cref{theorem:skb_not_actacc}.
The third and last section can be read independently of the second and contains an analogous result for the category of post-Lie algebras. 
The latter were introduced in \cite{Val07} and have found fruitful applications in numerical integration~\cite{CuEbMu17} as well as in the study of the Yang--Baxter equation, thanks to the fact that they are the infinitesimal objects associated with (smooth) skew braces, in the sense of \cite{BaiGuoLiShengTang24,Trap24x}.

\section{Semi-abelian varieties}
\label{sec:semiabelian}

In the present section, we state the definition of a semi-abelian variety in the sense of Universal Algebra following its characterization given in \cite[Theorem 1.1]{BoJa03}.
Moreover, even though our interest lies in this particular context, we set up the language and notions that are necessary for the rest of the paper, providing appropriate statements for general semi-abelian categories whenever possible.
For a comprehensive treatment, we refer the reader to 
\cite{JanMarTho02}.

\smallskip

\begin{definition}\label{def:semiabelian-2}
    A pointed category $\C$ with finite products and finite coproducts is \emph{semi-abelian} if $\C$ is Barr-exact and Bourn-protomodular.
\end{definition}

\smallskip

A category $\C$ is Barr-exact if it is regular and all its equivalence relations arise from a kernel pair.
Moreover, as shown in \cite{JanMarTho02}, under the condition that a Barr-exact category $\C$ has a zero object, Bourn-protomodularity is equivalent to the validity of the so-called split short five lemma.
Furthermore, if $\C$ further admits all finite coproducts (sums), then there is an equivalence between the split short five lemma and the following property: for all split short exact sequences
\[\begin{tikzcd}[ampersand replacement=\&]
	X \& E \& A
	\arrow["k", tail, from=1-1, to=1-2]
	\arrow["\varepsilon", shift left, two heads, from=1-2, to=1-3]
	\arrow["s", shift left, from=1-3, to=1-2]
\end{tikzcd}\]
the 
pair $(k,s)$ is \emph{strongly jointly epimorphic}.
The exactness of the above sequence refers to the fact that $k = \ker\varepsilon$ and $\varepsilon:E \to A$ is a regular epimorphism, i.e., the coequalizer of a pair of parallel morphisms.

\smallskip

When dealing with a variety $\V$ in the sense of Universal Algebra, to verify that $\V$ is semi-abelian it is enough to prove that $\V$ admits a sufficient number of \emph{terms} satisfying certain properties.

In general, consider a set $\F$ of basic operation symbols for some universal algebra, each with its corresponding arity.
Let $X$ be a set of \emph{variables} that can be assumed to be disjoint from $\F$.

The set of \emph{terms of type $\F$ over $X$} is the smallest set $T(X)$ containing $X$ as well as the constants in $\F$ and such that, if $\alpha_1, \ldots, \alpha_n \in T(X)$ and $\omega \in \F$ is an $n$-ary function symbol, then the string $\omega(\alpha_1, \ldots, \alpha_n)$ is in $T(X)$.
For each term $\alpha \in T(X)$, it is possible to assign an arity $n\in \N_0$, where $n$ is the number of variables that occur explicitly in $\alpha$.
For example, if $\F = \{+\}$, then $x$, $x + y$ and $x + (x + y)$ are terms of type $\F$ over $X = \{x,y\}$.

\begin{definition}[{cf. Theorem 1.1 in \cite{BoJa03}}]
\label{definition:semiabelian_variety}    
A variety $\V$ is \emph{semi-abelian} if it admits a unique constant $0$, a set of binary terms $t_1, \ldots, t_n$ for some $n \in \mathbb{N}$, together with an $(n+1)$-ary term $t$ such that the identities
    \begin{align*}
        t_i(x,x) = 0
        \qquad\text{and}\qquad
        t\left(x, t_1(x,y),\, \ldots, t_n(x,y)\right) = y \,,
    \end{align*}
    hold in $\V$, for all $i = 1, \ldots, n$.
\end{definition}

\begin{example}
    Many classical examples of varieties such as groups and Lie algebras form semi-abelian varieties.
    Reflecting our interest, the variety $\Skb$ of skew braces is also semi-abelian.
    This fact was already noted in \cite{FacPom24} and can be directly verified through \cref{definition:semiabelian_variety} by considering the terms 
    \[
    t_1(x,y) := -x + y
    \qquad\text{and}\qquad
    t(x,y) := x+y
    \]
    in $\Skb$ together with its unique constant $0$.
\end{example}

\subsection{Split extensions and action accessibility}

Let $\C$ be a semi-abelian category and let $A,X$ be objects of $\C$.
In \cite{BorJanKel05}, the authors introduced a natural notion of an \emph{internal action} of $A$ on $X$ yielding a functor $\Act(\boldsymbol{-},X):\C^{\text{op}} \to \Set$ 
mapping an object $A$ to the set $\Act(A,X)$ of $\C$-actions of $A$ on $X$.
An important feature of internal actions is that they are naturally classified by split extensions.

\begin{definition}\label{definition:short_ex}
    Let $\C$ be a semi-abelian category and let $A,X$ be objects of~$\C$.
    A \emph{split extension of $A$ by $X$} is a quadruple $(E,p,s,\mu)$ in $\C$ presented by a diagram of the form
    \[
    \Sigma:
    \begin{tikzcd}[ampersand replacement=\&]
	X \& E \& A
	\arrow["\mu", from=1-1, to=1-2]
	\arrow["p", shift left, from=1-2, to=1-3]
	\arrow["s", shift left, from=1-3, to=1-2]
    \end{tikzcd}\]
    with $\mu = \ker(p)$ and $ps = \id_A$.
    In other words, a split extension is a split short exact sequence $X\to E\to A$ in $\C$
    together with a fixed section $s:A \to E$.
\end{definition}

A morphism $(E,p,s,\mu) \to (E',p',s',\mu')$ between two  
such diagrams is a commutative diagram of the form 
\[\begin{tikzcd}
	X & E & A \\
	{X'} & {E'} & {A'}
	\arrow["\mu", from=1-1, to=1-2]
	\arrow["g"', from=1-1, to=2-1]
	\arrow["p", shift left, from=1-2, to=1-3]
	\arrow["f"', from=1-2, to=2-2]
	\arrow["s", shift left, from=1-3, to=1-2]
	\arrow["h", from=1-3, to=2-3]
	\arrow["{\mu'}"', from=2-1, to=2-2]
	\arrow["{p'}", shift left, from=2-2, to=2-3]
	\arrow["{s'}", shift left, from=2-3, to=2-2]
\end{tikzcd}\]
If we fix $X \in \C$ (resp. $A \in \C$) then $g:X \to X'$ (resp. $h:A \to A'$) in the latter definition is assumed to be the identity morphism.
In particular, if both $X$ and $A$ are fixed then $f:E \to E'$ is necessarily an isomorphism, thanks to the split short five lemma.

We denote by $\Split(A,X)$ the set of isomorphism classes of $A$ by $X$.
Moreover, if we only fix $X \in \C$ and let $A$ be arbitrary, then the above definitions yield a contravariant functor $\Split(\boldsymbol{-},X):\C^{\text{op}} \to \Set$  
such that, for all morphisms $f:A'\to A$, the function $\Split(f,X)$ maps (the equivalence class of) a split extension $(E,p,s,\mu)$ to the (equivalence class of the) split extension $(E',p',s',\mu')$ defined by requiring that

\[\begin{tikzcd}[ampersand replacement=\&]
	{E'} \& {A'} \\
	E \& A
	\arrow["{p'}", from=1-1, to=1-2]
	\arrow["g", from=1-1, to=2-1]
	\arrow["f", from=1-2, to=2-2]
	\arrow["p", from=2-1, to=2-2]
\end{tikzcd}\]
\noindent
displays a pull-back and that $s':A' \to E'$, $\mu':X\to E'$ are the unique morphisms such that $p'\mu' = 0$, $\mu = g\mu'$, $p's' = \id_{A'}$ and $gs' = sf$ hold. 
The existence and uniqueness of $s',\mu'$ indeed follow from the fact that the above square is a pull-back. 

\smallskip

Let us now turn our attention to 
\cite[Definition 3.5]{BorJanKel05}.

\begin{definition}\label[definition]{def:action_rep_cat}
    A semi-abelian category $\C$ has \emph{representable internal actions} if, for all objects $X\in \C$, there exists an object $T(X)\in\C$  and a natural isomorphism
    \begin{align}\label{eq:action_rep}
        \eta_X:\Split(\boldsymbol{-}\,,X) \xrightarrow{\sim} \C(\boldsymbol{-}\,,T(X)) \,.
    \end{align}
\end{definition}

A semi-abelian category that has representable internal actions is said to be \emph{action representable}.
If $X$ is an object of such a category, then the object $T(X)$ is called the \emph{actor} of $X$ while a morphism $B \to T(X)$ in the image $\eta_X(B)$ takes the name of \emph{acting morphism}, for all objects $B$.

Action-representable categories belong to a wider class of ``well-behaved'' categories satisfying the so-called notion of \emph{action accessibility}, as defined in \cite{BoJa09} and which we state in what follows.

Before doing this, given an object $X$ in a semi-abelian category $\C$, let us denote by $\Split(X)$ the category of (isomorphism classes of) split extensions in $\C$ with fixed kernel $X$ (in the sense of \cref{definition:short_ex} and the subsequent comments). 
Recall that an object $\Sigma$ in $\Split(X)$ is called \emph{faithful} if all objects in $\Split(X)$ admit at most one morphism into $\Sigma$.
Moreover, an object in $\Split(X)$ is \emph{accessible} if it admits a morphism into a faithful object of $\Split(X)$, i.e., into a $\Sigma$ as above.
For more details we refer the reader to \cite{BoJa09}.

\begin{definition}
    A semi-abelian category $\C$ is called \emph{action accessible} if every object of $\Split(X)$ is accessible, for all $X \in \C$.
\end{definition}

\begin{remark}\label[remark]{remark:implication}
    A semi-abelian, action-representable category $\C$ is action accessible as a consequence of 
    \cite[Theorem 6.3]{BorJanKel05}.
    Indeed, in its terms, given an object $X \in \C$, we obtain a natural faithful object in $\Split(X)$ as the split extension 
    \[\begin{tikzcd}
	X & {X \rtimes T(X)} & {T(X)}
	\arrow["\mu", tail, from=1-1, to=1-2]
	\arrow["p", two heads, from=1-2, to=1-3]
    \end{tikzcd}\]
    where $T(X)$ is the actor of $X$ while $\mu$, $p$ and the section $T(X) \to X \rtimes T(X)$ are appropriate canonical morphisms.
\end{remark}

\begin{example}
    Categories of interest in the sense of Orzech \cite{Orz72I} are particular instances of action-accessible categories, as was proved in \cite{Mo10}.
    Not necessarily unital rings and more generally algebras over a commutative unital ring form examples of categories belonging to this class.
\end{example}

Action accessibility is a property 
inherited by the Birkhoff subcategories of an action-accessible category, as shown by \cite[Proposition 2.3]{BoJa09}. 
In particular, if we restrict our interest to varieties in the sense of Universal Algebra, then we can state the aforementioned result in the following manner.

\begin{theorem}\label[theorem]{lemma:sub_action}
    If $\V$ is a semi-abelian and action-accessible variety, then all subvarieties of $\V$ are action accessible.
\end{theorem}

\section{Action accessibility of skew braces}

The aim of this section is to show that the variety $\Skb$ of skew braces is not action accessible and therefore not action representable.
We will proceed by first stating some basic properties of skew braces and then by recalling the notion of \emph{centralizer} of a sub-object in a semi-abelian category, essentially due to \cite{Huq68}.
Our main argument relies on \cite[Corollary 2.6]{CiMan12} and will be laid out by providing an example of a skew brace admitting a proper ideal with a non-normal centralizer.

\subsection{Basics on skew braces}

This brief subsection introduces fundamental notions and fixes notation related to skew braces. 
We refer the reader to \cite{CeSmVe19,GuVe17} for further details.

\smallskip

\begin{definition}
    A skew (left) brace is a set $B$ equipped with two group operations $+$ and $\circ$ satisfying the following identity:
    \begin{align}\label{eq:skewbrace-identity}
        a \circ (b+c) = a \circ b - a + a \circ c \,,
    \end{align}
    for all $a,b,c \in B$.
\end{definition}

It is a routine computation to verify that the identity of the multiplicative group $(B,\circ)$ coincides with the identity $0$ of the additive group $(B,+)$. As a consequence, the equality \eqref{eq:skewbrace-identity} can also be written as
\begin{align}\label{eq:skewbrace-identity-2}
    a \circ (b+c) = a\circ b + a\circ(a^-+c),
\end{align}
for all  $a,b,c\in B$, where by $a^-$ we denote the inverse of $a$ with respect to the operation $\circ$, for all $a\in B$.
\smallskip

Note that, if $\left(B,+\right)$ is a group, then $\left(B,+, +\right)$ and $\left(B,+, +^{op}\right)$ are skew braces that are called the \emph{trivial} and the \emph{almost trivial skew brace} on the group $\left(B,+\right)$, respectively. 
The first proper examples can be found  in \cite{GuVe17} and \cite{SmVe18}. 
Moreover, for explicit classifications of skew braces of small orders, the reader is referred to \cite{BaNeYa20}, while a wide variety of techniques for structural constructions has been extensively developed throughout the literature.
        
\medskip

Let $(B,+,\circ)$ be a skew brace and, for all $a\in B$, define a map $\lambda_a:B \to B$ by setting $\lambda_a(b) = -a + a \circ b$, for all $b \in B$. The importance of such maps lies in the following basic properties.

\begin{proposition}\label[proposition]{proposition:lambda}
    Let $(B,+,\circ)$ be a skew brace. Then the following hold:
    \begin{itemize}
        \item[{\rm (1)}]  $\lambda_a \in \Aut(B,+)$, for all $a\in B$;
        \item[{\rm (2)}] the map $\lambda \colon (B,\circ) \to \Aut(B,+),\, a \mapsto \lambda_a$ is a group homomorphism.
    \end{itemize}
\end{proposition}

Since the class of all skew braces forms a variety, usual notions from Universal Algebra apply in this context.
For instance, a sub-skew brace of a skew brace $(B,+,\circ)$ is an additive subgroup $I\leq (B,+)$ which is closed under the $\circ$ operation. 

\smallskip

\begin{definition}
An additive subgroup~$I$ of a skew brace~$(B,+,\circ)$ is said to~be:
\begin{itemize}
    \item[{\rm(1)}] a \emph{left ideal} if $I$ is $\lambda$-invariant, i.e., such that $\lambda_a(I) \subseteq I$, for all $a\in B$;
    \item[{\rm(2)}] an \emph{ideal} of $B$ if $I$ is a left ideal such that $(I,+)$ is a normal subgroup of $(B,+)$ and $(I,\circ)$ is a normal subgroup of $(B,\circ)$.
    
\end{itemize}
\end{definition}

\smallskip

Consequently, if $I$ is an ideal of $B$, then there is a well-defined map
\begin{align*}
    \lambda^I: (B,\circ) \to \Aut(I,+)
\end{align*}
defined by restricting $\lambda_b$ to $I$, i.e., by setting $\lambda^I(b):= (\lambda_b)_{\mid_{I}}$, for all $b\in B$.
One can easily check that the equivalence
\begin{align*}
    b\in \ker\lambda^I\ \iff\  \forall\ x\in I\quad b \circ x = b + x\,
\end{align*}
is satisfied.

\medskip

As it is clear from the first definition of this section, it is possible to develop an analogous theory for skew \emph{right braces} in which an identity dual to \eqref{eq:skewbrace-identity} holds.
In particular, a natural class of skew braces that arises through this elementary consideration is that of \emph{two-sided} skew braces.

\begin{definition}
    A skew brace $(B,+,\circ)$ is \emph{two-sided} if the identity
    \begin{align}\label{eq:twosided}
        (a + b) \circ c = a \circ c - c + b \circ c 
    \end{align}
    holds, for all $a,b,c \in B$.
\end{definition}

Note that if $R$ is a Jacobson radical ring, then the structure $(R,+,\circ)$ is a two-sided brace, where $\circ$ is the classical \emph{adjoint operation} defined by $a\circ b:= ab + a + b$, for all $a,b\in R$. 
Two-sided skew braces and their structure were investigated in different works such as \cite{Nas19,Trap23,Ts26} and have often proven to be an optimal ground to address important brace-theoretic questions.
Moreover, in \cite{MazRybSte25} the authors introduced an obstruction for two-sidedness in the wider class of dual weak braces and linked it to the study of solutions to the set-theoretic Yang--Baxter equation.

\medskip
Let $(B,+,\circ)$ be a skew brace.
For all $a \in B$, let $\sigma_a:B \to B$ be the map defined by setting 
\begin{align*}
    \sigma_a(b):= b \circ a - a \,,
\end{align*}
for all $b \in B$. 
It is straightforward to verify that $\sigma_a \in \Aut(B,+)$ if and only if $B$ is a two-sided skew brace. In this case, one also has that $\sigma \colon (B,\circ) \to \Aut(B,+)\,, a \mapsto \sigma_a$
is a group anti-homomorphism.

\begin{remark}\label[remark]{remark:sigma-ideal}
Let $(B,+,\circ)$ be a skew brace and consider an additively and multiplicatively normal subgroup $I$ of $B$.
Then $I$ is an ideal of $B$ if and only if it is $\sigma$-invariant in the sense that $\sigma_a(I) \subseteq I$, for all $a \in B$.
This follows from the fact that 
\begin{align*}
    \sigma_a(x) = a + \lambda_a(a^-\circ x\circ a) -  a\in I \,,
\end{align*}
for all $a \in B$ and $x \in I$.
\end{remark}



\smallskip

\subsection{Centralizers and action accessibility}

\begin{definition}\label[definition]{def:cooperate}
    Let $\C$ be a pointed category with binary products and let $f: X \to A$, $g: Y \to A$ be two morphisms in $\C$ sharing the same codomain.
    We say that $f$ and $g$ \emph{cooperate} in $A$ if there exists a morphism $\psi:X \times Y \to A$ making the following diagram commute:
    \[\begin{tikzcd}[ampersand replacement=\&,cramped]
    	X \& {X \times Y} \& Y \\
    	\& A
    	\arrow["{\{\id_X,\,0\}}", from=1-1, to=1-2]
    	\arrow["f"', from=1-1, to=2-2]
    	\arrow["\psi", dashed, from=1-2, to=2-2]
    	\arrow["{\{0,\id_Y\}}"', from=1-3, to=1-2]
    	\arrow["g", from=1-3, to=2-2]
    \end{tikzcd}\]
    where $\{\id_X,0\}$ is the unique morphism into $X \times Y$ making $\id_X$ and $0_{X,Y}$ factor through the projections $\pi_X$ and $\pi_Y$, respectively. Analogous considerations hold for $\{0_{Y,X}, \id_Y\}$. 
    If $f:X \to A$ and $g:Y \to A$ cooperate and $\C$ is unital, the morphism $\psi:X \times Y \to A$ is unique and takes the name of the \emph{cooperator} of $f$ and $g$ in $A$.
\end{definition}

In what follows, we are interested in a specialization of the above definition to sub-objects.
In detail, recall that for all objects $A$ in a category $\C$ it is possible to define a partial order on the (class of) monomorphisms $\mu:X \rightarrowtail A$, $\nu:Y \rightarrowtail A$ into $A$ by setting $\mu \leq \nu$ if there exists a morphism $\varphi:X \to Y$ such that $\mu = \nu\varphi$.
A sub-object of $A$ is an equivalence class of monomorphisms $\mu,\nu$ into $A$ such that $\mu \leq \nu$ and $\nu \leq \mu$.

In this context, if $X$ and $Y$ are sub-objects of an object $A$ in $\C$, then we will say that $X$ and $Y$ cooperate if their respective inclusions $X \rightarrowtail A$, $Y \rightarrowtail A$ cooperate in the sense of the above definition.

\begin{definition}\label[definition]{def:centralizer}
    Let $\C$ be a pointed category with finite products, let $A$ be an object of $\C$ and consider a sub-object $X$ of $A$.
    The largest sub-object of $A$ cooperating with $X$, if it exists, is called the \emph{centralizer} of $X$ in $A$ and is denoted by $Z_A(X)$.
\end{definition}

\begin{example}
    If $G$ is a group, then a sub-object of $G$ is simply identified with the structure of a subgroup of $G$.
    As a consequence, two subgroups $X,Y \leq G$ cooperate in $G$ if and only if $xy = yx$ holds, for all $x \in X$ and $y \in Y$.
    In particular, the centralizer of $X$ in $G$ coincides with its usual group-theoretic centralizer.
    Analogously, if $A$ is a skew brace, then a sub-object $X \rightarrowtail A$ of $A$ is (the equivalence class of the inclusion of) a sub-skew brace of $A$. 
\end{example}

The following result is of fundamental importance in our argument and justifies the use of centralizers.
For more details, we refer the reader to \cite[Corollary 2.6]{CiMan12}.

\begin{proposition}\label[proposition]{proposition:condition_actionacc}
    Let $A$ be an object of a semi-abelian category $\C$ and let $X$ be a normal sub-object of $A$.
    If $\C$ is action accessible, then $X$ admits a centralizer in $A$ and the latter is normal.
\end{proposition}

Let us turn our attention back to skew braces.
As shown in \cite{BouFacPom23}, a sub-skew brace $I$ of a skew brace $B$ is normal if and only if $I$ is an ideal of $B$.
It is now important to characterize the cooperating sub-objects in~$\Skb$.
To this end, for a given skew brace $B$ and one of its sub-skew braces~$I$, we denote by $C^+_B(I)$ and $C^{\circ}_B(I)$ the centralizer of $I$ in $B$ in relation to the additive $+$ and multiplicative $\circ$ operations, respectively.

\begin{lemma}\label[lemma]{lemma:cooperating_subskb}
    Let $B$ be a skew brace and let $I$ be an ideal of $B$.
    The centralizer of $I$ in~$B$ is the largest sub-skew brace of $B$ contained in the set
    \begin{align}\label{eq:max_centralizer}
        C(B,I):= C^+_B(I) \cap C^\circ_B(I) \cap \ker\lambda^{I} \,,
    \end{align}  
    if the latter exists.         
\end{lemma}

\begin{proof}
        Preliminarily, notice that $y \in C(B,I)$ if and only if the following holds:
        \[
            x+y = y+x = y \circ x = x \circ y \,,
        \]
        for all $x \in I$.
        Now, let $J$ be a sub-skew brace of $B$, assume it cooperates with $I$ and let us prove that $J \subseteq C(B,I)$.
        To this end, note that $J \subseteq C^+_B(I)\cap C^\circ_B(I)$ since $(J,+)$ and $(J,\circ)$ cooperate with $(I,+)$ and $(I,\circ)$, respectively.
        Consider now the cooperator $\psi:I \times J \to B$ of $I$ and $J$ as in \cref{def:cooperate} and let $x \in I$ and $y \in J$. Then we have
        \begin{align*}
            y \circ x 
            = \psi\left( (0,y) \circ (x,0) \right)
            = 
            \psi(x,y)
            = \psi\left( (0,y) + (x,0) \right)
            = y + x \,,
        \end{align*}
        from which it follows that $y \in \ker\lambda^I$.
        Therefore, $J \subseteq C(B,I)$.
        Conversely, if $J$ is a sub-skew brace of $B$ contained in $C(B,I)$, then it is straightforward to verify that the map $\psi\colon I \times J \to B$, $(x,y) \mapsto x+y$ is the cooperator of $I$ and $J$ in $B$, thus proving the assertion.
    \end{proof}

Note that if $I = B$, then $C(B,B)$ coincides with the \emph{annihilator} $\operatorname{Ann}(B)$ of $B$, which is an ideal of $B$; see \cite[Definition~7]{CaCoSt19}.
Despite this fact, it is worth remarking that the set $C(B,I)$ may generally fail to be a sub-skew brace, as the following example shows.

\begin{example}
    Consider the skew brace $(B,+,\circ)$ of 
    order $12$ as classified in \cite[Lemma 7.2]{AcBon22} 
    and obtained by setting $q = 3$.
    Explicitly, $(B,+) = \Z_3 \times \Z_4$ while the multiplication is defined by setting
    \begin{align*}
        (n,\,m) \circ (x,\,y) = \left(n + (-1)^{\frac{m(m-1)}{2}}x,\, m + (-1)^m y \right)
    \end{align*}
    for all $n,x \in \Z_3$ and $m,y \in \Z_4$.
    The socle $I = \Soc(B)$ of $B$ is an ideal of order $3$ and is explicitly given by $I = \{(n,0) \mid n \in \Z_3 \}$.
    It is straightforward to verify that $C(B,I)$ is the set of all pairs $(n,m)\in B$ that satisfy the identity 
    $$
    x\left( 1 - (-1)^{\frac{m(m-1)}{2}} \right) = 0,
    $$ 
    for all $x \in \Z_3$. Hence,
    \begin{align*}
        C(B,I) = \{ (0,0), (0,1), (1,0), (1,1), (2,0), (2,1) \},
    \end{align*}
    which is not a sub-skew brace of $B$ since it is not even additively closed.
    Indeed, for example, we have that $(2,1) + (2,1) = (1,2) \notin C(B,I)$.
    Despite this, it is not difficult to verify that $I$ coincides with its own centralizer, i.e., $I = Z_B(I)$.
\end{example}

We are now ready to prove the main result of this paper.
In particular, in what follows we make use of the classical \emph{semidirect product} of braces in terms of \cite[p.~113]{CedFerOkn14}.
In detail, if $B,H$ are braces and  $\varphi:(B,\circ) \to \Aut(H,+,\circ)$ is a group homomorphism, then there exists a skew brace structure on $H \times B$, denoted by $H \rtimes_\varphi B$, whose operations are defined by setting
\begin{align*}
    (a,u) + (b,v) = (a+ b, u+ v) 
        \qquad\text{and}\qquad
    (a,u) \circ (b,v) = (a\circ \varphi(b), u\circ v) \,,
\end{align*}
for all $(a,u),(b,v)\in H\times B$.

\begin{theorem}\label[theorem]{theorem:skb_not_actacc}
    The category $\Skb$ is not action accessible.
    \begin{proof}
        First, note that, thanks to \cref{lemma:sub_action}, we can reduce to proving that the subvariety of all braces is not action accessible.
        To this aim, it is sufficient to apply \cref{proposition:condition_actionacc} and construct a brace admitting an ideal whose centralizer exists as in \cref{def:centralizer} but is not an ideal itself.

        The candidate is the brace $B$ of order $24$ and ID equal to $625$ in the \texttt{GAP YangBaxter} database \cite{GAP4,YangBaxter} and built as a semidirect product as follows.
        Let $(Q,+,\circ)$ be the unique skew brace of order $8$ whose additive group is $\Z_2 \times \Z_4$ and with trivial socle, as classified in \cite{Bach15}.
        Specifically, $(Q,\circ) \cong D_8$ and its operation is defined by setting 
        \begin{align*}
            \begin{pmatrix}
                a \\
                x + 2y
            \end{pmatrix} \circ
            \begin{pmatrix}
                b \\
                u + 2v
            \end{pmatrix} =
            \begin{pmatrix}
                a + b + u(a + x + y + ax) \\
                x + 2y + 2xb + 2(a + xy)u + u + 2v
            \end{pmatrix}
        \end{align*}
        for all $a,x,y$ and $b,u,v \in \{0,1\}$.
        As a skew brace, $Q$ is generated by the elements
        $s := \mybinom{1}{0}$ 
        and
        $t := \mybinom{0}{3}$, 
        both multiplicative involutions such that $s \circ t$ has multiplicative order equal to $4$.
        
        View $J := \Z_3$ as a trivial brace and consider the group homomorphism 
        \[
            \varphi:(Q,\circ) \to \Aut(J), \, q \mapsto 
            \varphi(q):= \varphi_q
        \]
        defined by setting $\varphi_s(x) = \varphi_t(x) = -x$, for all $x\in J$, i.e., by mapping $s$ and $t$ to the unique non-trivial automorphism of $J$. 
        Notice that \begin{align*}
        \left(\ker\varphi,\circ\right) =
            \langle s \circ t \rangle_\circ =
            \Big\langle\mybinom{0}{1}\Big\rangle_\circ 
            =\Big\{\mybinom{0}{0},\,\mybinom{0}{1},\, \mybinom{1}{2},\, \mybinom{1}{1}\Big\}\,.
        \end{align*}
        Besides, the  
        semidirect product $B := J \rtimes_\varphi Q$ of $J$ and $Q$ via $\varphi$
        is a brace of order $24$ such that $(B,+) \cong \Z_2 \times \Z_{12}$ and $(B,\circ) = \Z_3 \rtimes_{\varphi} D_8 \cong D_{24}$.
        In particular, it is the unique such skew brace with socle of order $3$. 
        Set $I := \Soc(B)$ and note that
        \begin{align*}
            I = \Soc(J)\times \Soc(Q) 
            = J \times \langle 0 \rangle \,,
        \end{align*}
        which is isomorphic to $J$ as a brace. 
        Moreover, it is not difficult to verify that 
        \begin{align*}
        \ker\lambda^I = J \times \ker\varphi = C^\circ_B(I).
        \end{align*}
        Consequently, $C(B,I) = J \times \left\langle
        \mybinom{0}{1}\right\rangle_\circ$ which has cardinality $12$.
        Let us now observe that $C(B,I)$ is not a sub-brace of $B$ since 
        $$
        \Big(0,\mybinom{0}{1} \Big)
        +
        \Big(0,\mybinom{0}{1} \Big) =
        \Big(0,\mybinom{0}{2} \Big)
        \notin C(B,I) \,.
        $$ 
        Moreover, 
        it is straightforward to verify that 
        $$
        J \times \left\langle (s \circ t)^2 \right\rangle = 
        J\times\Big\{
        \mybinom{0}{0},\, 
        \mybinom{1}{2}\Big\}
        $$ 
        is the largest sub-brace contained in $C(B,I)$ and thus coincides with the centralizer $Z_B(I)$ of $I$ in $B$.
        Indeed, by simply listing all elements of $C(B,I)$, it is possible to note that the additive multiples of the elements of the form $(x,\binom{0}{1})$ and $(x,\binom{1}{1})$ with $x \in J$ are not contained in $C(B,I)$.
        The remaining elements, i.e., those of the form $(y,\binom{0}{0})$ and $(y,\binom{1}{2})$ with $y \in J$, are exactly the ones that constitute the above candidate; hence the claim follows.  
        
        Finally, to
        show that $Z_B(I)$ is not an ideal of $B$, it is enough to verify that $\left\langle (s \circ t)^2 \right\rangle$ is not an ideal of $Q$.
        To this aim, if we set $p = \mybinom{0}{1} \in Q$, then we have that
        \begin{align*}
            \lambda_p\left( (s \circ t)^2 \right) =
            \lambda_{p}\left( \mybinom{1}{2}\right) = 
            \mybinom{1}{0}
            \notin
            \left\langle (s \circ t)^2 \right\rangle \,,
        \end{align*}
        thus terminating the proof.
    \end{proof}
\end{theorem}

\begin{remark} \,
\begin{itemize}
    \item[(1)] 
    Simple \texttt{GAP} calculations point out that an ideal $I$ of a skew brace $B$ of order up to $23$ always admits a normal centralizer, even though it can be properly contained in the set $C(B,I)$ in \eqref{eq:max_centralizer}.

    \item[(2)]
    A result analogous to \cref{theorem:skb_not_actacc} can be obtained with the skew brace $B$ of order $24$ and ID equal to $435$ in the database implemented in the \texttt{GAP} \textsf{YangBaxter} package.
    More specifically, while the additive group of $B$ is non-abelian and isomorphic to the direct product $\Z_2 \times \text{Dic}_{12}$, where $\text{Dic}_{12}$ is the dicyclic group of order $12$, the multiplicative group of $B$ is isomorphic to $\Z_2 \times \Z_2 \times \Sym_3$.
    In particular, $B$ has an ideal of order $3$ that admits a non-normal centralizer.
    \end{itemize}
\end{remark}

\begin{remark}
    Normal cooperating sub-skew braces were investigated by the authors of \cite{BouFacPom23}, who showed that an ideal $I$ in a skew brace $B$ always admits a largest ideal $K$ cooperating with $I$ in $B$.
    Our example shows that, in general, $K$ may fail to be the centralizer of $I$ in $B$ as in \cref{def:centralizer}.
    For instance, if $B = J \rtimes_{\varphi} Q$ is the brace of order $24$ contained in \cref{theorem:skb_not_actacc} and if $I = \Soc(B) = J \times \langle 0 \rangle$, then $I$ is the largest ideal of $B$ centralizing $I$ itself, even though $Z_B(I)$ is properly larger.
\end{remark}

\medskip
We close this section by showing that, by \cref{lemma:cooperating_subskb}, the Huq centralizer of an ideal in a two-sided skew brace coincides with its centralizer in the sense of \cite{BouFacPom23}.

\begin{proposition} 
Let $(B,+,\circ)$ be a two-sided skew brace and  
$I$ an ideal of $B$.
Then the set $C(B,I)$ defined in \eqref{eq:max_centralizer} is an ideal of $B$.
\end{proposition}
\begin{proof}
    Let $a,b\in C(B,I)$ and $x \in I$.
    Naturally, $a-b \in C^+_{B}(I)$.
    Moreover, we have that
    \begin{align*}
        (a-b) \circ x &=
        a \circ x -x + (-b) \circ x 
        =
        a \circ x - x + x - b \circ x + x 
        =
        a-b+x \,,
    \end{align*}
    as well as 
    $$
    x \circ (a-b) = x \circ a - x + x \circ (-b) = a + x - x \circ b +x = a-b+x,
    $$ 
    thus showing that $a-b \in C^\circ_{B}(I)$.
    From the same calculations, we deduce that 
    $$
    \lambda_{a-b}(x) = b-a + (a-b) \circ x = x,
    $$ 
    hence that $a-b\in\ker\lambda^I$.
    This shows that $C(B,I)$ is an additive subgroup of~$B$.

    Now, to prove that $C(B,I)$ is a multiplicative subgroup of $B$, note that
    \begin{align*}
        (a \circ b^- + x) \circ b 
        &= a -b+ x \circ b
        = a + x
        = x \circ b -b + a \circ b^- \circ b
        = (x + a \circ b^-) \circ b
    \end{align*}
    thus ensuring that $a\circ b^- \in C^+_B(I)$.
    Besides, $C^\circ_B(I)$ is trivially a multiplicative subgroup of $B$. Lastly, by \cref{proposition:lambda}, we obtain that $\lambda_{a \circ b^-}(x) = \lambda_a\lambda^{-1}_b(x) = x$, i.e., $a\circ b^- \in \ker\lambda^I$.
    Therefore, $C(B,I)$ is a 
    sub-skew brace of $B$.
    Let us now prove that 
    $C(B,I)$ is $\lambda$-invariant.
    To this aim, let $h \in B$, $a \in C(B,I)$ and first note that
    \begin{align*}
        \lambda_h(a) + x =
        \lambda_h\left( a + \lambda_h^{-1}(x) \right) =
        \lambda_h\left( \lambda_h^{-1}(x) + a \right) =
        x + \lambda_h(a) \,,
    \end{align*}
    holds, for all $x \in I$, thus showing that $\lambda_h(a)\in C^+_B(I)$. 
    Furthermore, we have that
    \begin{align*}
        \lambda_h(a) \circ x =
        h \circ \left( h^- + a \right) \circ x =
        h \circ \left(h^- \circ x - x + a \circ x \right) =
        x + \lambda_h(a) \,,
    \end{align*}
    proving that $\lambda_h(a)\in\ker\lambda^I$.
    Finally, 
    \begin{align*}
        x \circ \lambda_h(a) &=
        x \circ h \circ \left(h^{-} + a\right) =
        h \circ \left(h^{-} \circ x \circ h\right) \circ \left(h^{-} + a \right) \\
        &=
        h \circ \left(y \circ h^{-} - y + y \circ a\right) \,,
    \end{align*}
    where we 
    set $y:= \left(h^{-} \circ x \circ h\right)\in I$. It follows that
    \begin{align*}
        x \circ \lambda_h(a) =
        h \circ \left(y \circ h^{-} + a \right) =
        x - h + h \circ a =
        x + \lambda_h(a) \,,
    \end{align*}
    showing that $\lambda_h(a)\in C^\circ_B(I)$.
    Therefore, $C(B,I)$ is $\lambda$-invariant.

    It remains to verify  that~$C(B,I)$ 
    is a normal subgroup of $(B,+)$ and $(B,\circ)$, respectively.
    To this end, let~$h\in B$, $a\in C(B,I)$ and first note that obviously~$-h+ a + h\in C^+_B(I)$. 
    It follows that
    \begin{align*}
        \left(-h +a +h\right) \circ x 
        &=
        (-h) \circ x -x + a \circ x -x + h \circ x 
        = x - h \circ x + a + h \circ x \\[0.1cm]
        &= x - \lambda_h(x) - h + a + h + \lambda_h(x) 
        = -h + a + h + x \,,
    \end{align*}
    namely, $-h+ a + h \in \ker\lambda^I$.
    To show that $-h+ a + h \in C^\circ_B(I)$, it is sufficient to note that
    \begin{align*}
        x \circ (-h + a + h) &=
        x \circ (-h) -x + x \circ a -x + x \circ h 
        =
        x - x \circ h + a + x \circ h \\[0.1cm]
        &=
        x - h - \sigma_h(x) + a + \sigma_h(x) + h 
        =
        x - h + a + h \,,
    \end{align*}
    since, by \cref{remark:sigma-ideal}, $\sigma_h(x)\in I$.
    Let us  
    lastly check that $h^- \circ a \circ h \in C(B,I)$.
    Naturally, $h^- \circ a \circ h$ is already contained in $C^\circ_B(I)$. Furthermore, if we notice that 
    \begin{align*}
        h^{-} \circ a \circ h + x 
        &= h^{-} \circ a \circ h \circ \lambda_h^{-1}\lambda_a^{-1}\lambda_h^{}(x)
        = h^{-} \circ a \circ h \circ x \,,
    \end{align*}
    then we obtain $h^-\circ a \circ h \in \ker\lambda^I$, 
    while thanks to
    \begin{align*}
        h \circ (x + h^- \circ a \circ h) &=
        h \circ x -h + a \circ h
        = (h \circ x \circ h^- + a) \circ h 
        = (a + h \circ x \circ h^-) \circ h \\
        &= a \circ h -h + h \circ x \\
        &= h \circ (h^- \circ a \circ h + x) \,,
    \end{align*}
    we obtain $h^{-} \circ a \circ h\in
    C^+_B(I)$,
    thereby proving the assertion.
    \end{proof}

\smallskip
\begin{question}
    Is the variety of two-sided skew braces action accessible?
    More generally, if $\V$ is an action-accessible subvariety of $\Skb$ and $B \in \V$, is $B$ two-sided?
\end{question}

\section{Action representability of post-Lie algebras}

In this section, we show that the category $\PostL_k$ of post-Lie algebras (over a fixed field $k$) is not action representable.
In particular, we prove the stronger statement according to which the subvariety $\PreL_k$ of pre-Lie algebras is not action accessible. 
For more details on these structures, we refer the reader to recent papers \cite{Bai20,BuDeMi24} and the references therein.

\medskip

Let us begin by first recalling a useful characterization of action accessibility for an \emph{operadic variety}, namely, 
a variety of non-associative algebras over a field $k$ determined by a set of multilinear identities (see \cite{ReVanCo23}).
Specifically, as proved in \cite{GarVan19} and pointed out in \cite{BroGarManVan25}, we have the following rather useful result.

\begin{proposition}\label[proposition]{proposition:action_acc}
    Let $\V$ be an operadic variety of non-associative algebras over a field $k$. Then the following are equivalent:
    \begin{enumerate}
        \item[{\rm(1)}] $\V$ is action accessible;
        \item[{\rm(2)}] There exist scalars $\lambda_1, \ldots, \lambda_8$, $\mu_1, \ldots, \mu_8 \in k$ such that the following are identities in $\V$:
        \begin{align}
            x(yz) = \lambda_1(xy)z &+ \lambda_2(yx)z + \lambda_3z(xy)+ \lambda_4z(yx) + \lambda_5(xz)y \notag\\
            &+ \lambda_6(zx)y + \lambda_7y(xz) + \lambda_8y(zx) \,,\label{eq:action_acc1} \\
            (yz)x = \mu_1(xy)z &+ \mu_2(yx)z + \mu_3z(xy)+ \mu_4z(yx) \notag\\
            &+ \mu_5(xz)y + \mu_6(zx)y + \mu_7y(xz)+ \mu_8y(zx) \label{eq:action_acc2}.
        \end{align}
    \end{enumerate}
\end{proposition}

\smallskip

Let us now introduce the main definition of this section.

\begin{definition}[see~\cite{Val07}]
    A \emph{post-Lie algebra} (over a field~$k$) is the datum of a~$k$-vec\-tor space $A$ together with two bilinear maps $\cdot$ and $\{\cdot,\cdot\}\colon A \times A \to A$ such that
    \begin{enumerate}
        \item[{\rm(1)}] $(A,\{\cdot,\cdot\})$ is a Lie algebra;
        \item[{\rm(2)}] The identities
            \begin{align}
                x \{y,z\} &= \{xy,z\} + \{y,xz\} \,, \label{eq:post1} \\
                \{x,y\} z &= a(x,y,z) - a(y,x,z) \,, \label{eq:post2}
            \end{align}
            hold for all $x,y,z \in A$,
            where $a(x,y,z):= x(yz) - (xy)z$ is the so-called \emph{associator} on $(A,\cdot)$.
    \end{enumerate}
    A \emph{pre-Lie algebra} is a non-associative algebra $(A,\cdot)$ such that
    \begin{align}\label{eq:preLie}
        x(yz) - (xy)z = y(xz) - (yx)z \,,
    \end{align}
    holds, for all $x,y,z \in A$, or, equivalently, such that
$        a(x,y,z) = a(y,x,z)$ 
    holds, for all $x,y,z \in A$.
\end{definition}

\begin{remark}\label[remark]{remark-PreLie-PostLie}\,
\begin{enumerate}
    \item[(1)] It is important to note that a pre-Lie algebra $(A,\cdot)$ is a particular post-Lie algebra $(A,\{\cdot,\cdot\},\cdot)$ with trivial bracket $\{\cdot,\cdot\}$.
    Hence, the variety $\PreL$ is a (proper) subvariety of $\PostL$. Despite this fact, pre-Lie algebras have attracted interest independently of post-Lie algebras (see \cite{Bai20} for a comprehensive treatment).
\item[(2)] Any pre-Lie algebra satisfies the identity \eqref{eq:action_acc1} in \cref{proposition:action_acc} (by considering $\lambda_1 = \lambda_7 = 1$, $\lambda_2 = -1$, and the other coefficients $\lambda_i = 0$).
\end{enumerate}

\end{remark}

\smallskip

\begin{examples} \,
    \begin{enumerate}
        \item[{\rm(1)}] An associative algebra is a pre-Lie algebra whose associator is identically zero, i.e., such that $a(x,y,z) = 0$, for all $x,y,z \in A$.
        \item[{\rm(2)}] A Lie algebra $(A,\{\cdot,\cdot\})$ always admits a structure of post-Lie algebra $(A,\{\cdot,\cdot\},\cdot)$ defined by setting $x \cdot y = 0$, for all $x,y \in A$.
        \item[{\rm(3)}] If $(A,\{\cdot,\cdot\})$ is a $2$-step nilpotent Lie algebra, then $(A,\{\cdot,\cdot\},\{\cdot,\cdot\})$ is a post-Lie algebra.
    \end{enumerate}
\end{examples}

Let us now turn our attention to the claim of this section.

\begin{proposition}\label[proposition]{proposition:pre_non_acc}
    Let $k$ be a field. The category 
    $\PostL_{k}$ of post-Lie algebras over $k$ is not action accessible.
    \begin{proof}
        In light of \cref{lemma:sub_action}, it is sufficient to prove the assertion for the variety $\PreL_{k}$ of pre-Lie algebras over $k$.
        Moreover, thanks to \cref{proposition:action_acc}, it is sufficient to display a pre-Lie algebra in which either \eqref{eq:action_acc1} or \eqref{eq:action_acc2} does not hold.
        To this aim, consider the four-dimensional pre-Lie algebra $I_4$ from \cite[Example 1.1]{Bu98}
        with basis $\{e_1,e_2,e_3,e_4\}$ and whose non-trivial products are given~by
        \begin{align*}
            e_1e_2 &= e_2 \,, \quad
            e_1e_3 = e_3 \,, \quad
            e_1e_4 = e_4\,, \quad
            e_1e_1  = 2e_1 \,,\quad
            e_2e_2 = e_3e_3 = e_4e_4 = e_1 \,.
        \end{align*}
        The term $(e_2e_2)e_3 = e_3$ does not belong to the subspace of $I_4$ generated by the elements
        \begin{align*}
            (e_3e_2)e_2\,, \ (e_2e_3)e_2\,, \ e_2(e_3e_2)\,, \ e_2(e_2e_3)\,,
        \end{align*}
        due to the fact that the latter terms are all equal to zero.
        Consequently, the identity \eqref{eq:action_acc2} cannot be satisfied in $A$ for any choice of $\mu_1, \ldots, \mu_8$.
        Therefore, we conclude that $\PreL_{k}$ is not action accessible.
    \end{proof}
\end{proposition}

As a consequence of \cref{proposition:pre_non_acc} and \cref{remark:implication},  we deduce that the category $\PostL_{k}$ of post-Lie algebras (over any field $k$) is not action representable.
A different conclusion regarding action accessibility can be obtained if we restrict our attention to Novikov algebras, i.e., to the class of those pre-Lie algebras $(A,\cdot)$ that satisfy $(xy)z = (xz)y$, for all $x,y,z \in A$.

\begin{proposition}\label[proposition]{proposition-Nov}
    The variety {\bf Nov}$_k$ of Novikov algebras {\rm(}over an arbitrary field $k${\rm)} is action accessible.
\end{proposition}
\begin{proof}
    By (2) in \cref{remark-PreLie-PostLie}, the identity \eqref{eq:action_acc1} is satisfied. Moreover, \eqref{eq:action_acc2} trivially holds (by considering as unique non-zero coefficient $\mu_2 = 1$).
    Therefore, the claim follows.
\end{proof}

Let $A$ be the pre-Lie algebra in the proof of \cref{proposition:pre_non_acc}. We underline that this counterexample heavily relies on the fact that $A$ is not Novikov since, for example, $(e_2e_2)e_3 \neq (e_2e_3)e_2$.
In general, the fact that all action-accessible varieties of non-associative algebras are varieties of Novikov algebras is not necessarily true.
This is justified by the fact that the variety {\bf Assoc$_k$} of associative algebras over a fixed field $k$ is action-accessible (e.g. set $\lambda_1 = 1$ and all other parameters to zero in \cref{proposition:action_acc}) but not all associative algebras are Novikov (e.g. consider the algebra $M_2(k)$ of $2 \times 2$ matrices over a field $k$).
Taking everything into consideration, it is natural to state the following question.

\begin{question}
    What are the action-accessible subvarieties of $\PreL_k$ beyond {\bf Assoc}$_k$ and {\bf Nov}$_k$?
\end{question}

\end{document}